\documentclass[12pt]{article}
\usepackage[dvips]{graphics}
\usepackage{amsmath}
\usepackage{amssymb}
\usepackage{amscd}
\usepackage{amsthm}
\usepackage{amsopn}
\usepackage{xspace}
\usepackage{verbatim}
\usepackage{amsmath}
\usepackage{amsfonts}
\usepackage{colortbl}
\usepackage{hyperref}
\usepackage{amsmath}
\usepackage{color}
\usepackage{srcltx}
\usepackage[utf8]{inputenc}
\newcommand{\R}{\mathbb{R}}

\newtheorem{theorem}{Theorem}
\newtheorem{proposition}{Proposition}[section]

\newtheorem{remark}{Remark}[section]

\newcommand\AMSname{AMS subject classifications}
\title{A minimizing problem of a polyharmonic operator with Critical Exponent  }
\author{Asma Benhamida\thanks{D\'epartement de Math\'ematiques,
Facult\'e des Sciences de Tunis,
Campus Universitaire, 2092 Tunis,
Universit\'e Tunis El Manar, Laboratoire Analyse non linéaire et Géométrie LR21LS08,
Tunisie.
E-mail: asma.benhamida2019@gmail.com}, $\,$  Rejeb Hadiji\thanks{Universit\'e Paris-Est Cr\'eteil, LAMA, Laboratoire d'Analyse et
de Math\'ematiques Appliqu\'ees, CNRS UMR 8050, UPEC, F-94010 Cr\'eteil, France. E-mail :
rejeb.hadiji@u-pec.fr } $\,$  and Habib Yazidi
\thanks{Universit\'e de Tunis, Ecole Nationale Sup\'erieure d'Ing\'enieurs de Tunis, 5 Avenue Taha Hssine, Bab Mnar 1008 Tunis,  Tunisie. E-mail : habib.yazidi@gmail.com} }
\begin{document}
\date{}

\maketitle
\begin{abstract}
In this work, we study the two following minimization problems for $r \in \mathbb{N}^{*}$,
\begin{equation*}
\begin{array}{ccc}
S_{0,r}(\varphi)=\displaystyle\inf_{u\in H_{0}^{r}(\Omega),\,\|u+\varphi\|_{L^{2^{*r}}}=1}\|u\|_{r}^{2}&
\textrm{and\,\,}&
S_{\theta,r}(\varphi)=\displaystyle\inf_{u\in H_{\theta}^{r}(\Omega),\,\|u+\varphi\|_{L^{2^{*r}}}=1}\|u\|_{r}^{2},
\end{array}
\end{equation*}
where $\Omega \subset \mathbb{R}^{N}, $ $N > 2r$, is a smooth bounded domain, $2^{*r}=\frac{2N}{N-2 r}$, $\varphi\in L^{2^{*r}} (\Omega) \cap C(\Omega)$ and the norm $\|. \|_{r}=\displaystyle{ \int_{\Omega} |(-\Delta)^{\alpha} .|^{2}dx}$ where $ \alpha=\frac{r}{2} $ if $r$ is even and $\|. \|_{r}=\displaystyle{ \int_{\Omega} |\nabla(-\Delta)^{\alpha} . |^{2}dx }$ where $\alpha = \frac{r-1}{2}$ if $r$ is odd.
Firstly, we prove that, when $\varphi \not\equiv 0, $ the infimum in $S_{0,r}(\varphi)$ and $S_{\theta,r}(\varphi)$ are attained.
Secondly, we show that $ S_{\theta,r}(\varphi)< S_{0,r}(\varphi) $ for a large class of $
\varphi$.

\medskip

\noindent Keywords : {Critical Sobolev exponent, poly-harmonic operator, minimizing problem.}
\medskip

\par
\noindent2010 \AMSname: 35J20, 35J25, 35H30, 35J60.
\end{abstract}

\section{Introduction and main results}
Let $r \in \mathbb{N}^{*}$ and $\Omega$ be a bounded domain of $\R^{N}$ with $N\geq 2r+1$. We define the space $H^{r}_{0}(\Omega)$ and $H^{r}_{\theta}(\Omega)$ by
$$H^{r}_{0}(\Omega):=\left\{ f\in H^{r}(\Omega) \mid D^{k}f=0 \,\textrm{on $\partial\Omega$ for $k=0,1\ldots, r-1$}\right\}$$
and
$$H^{r}_{\theta}(\Omega):=\left\{ f\in H^{r}(\Omega) \mid (-\Delta)^{k}f=0 \,\textrm{on $\partial\Omega$ for $0\leq k \leq [(r+1)/2]$}\right\}.$$
where $D^{k}f$ denote any derivative of order $k$ of the function $f$ and $[(r+1)/2]$ is the integer part of $(r+1)/2$.\\
Define the following norm
\begin{equation*}
\|f\|_{r}^{2}=\left\{\begin{array}{lll}
\displaystyle{ \int_{\Omega} |(-\Delta)^{r/2} f|^{2}dx} \quad &\textrm{if $r$ is even},\\
\displaystyle{ \int_{\Omega} |\nabla(-\Delta)^{\frac{r-1}{2}} f |^{2}dx } \quad &\textrm{if $r$ is odd}.
\end{array}
\right.
\end{equation*}

Then, we consider the following minimizing problem
\begin{equation}
S_{0,r}(\varphi)=\inf_{u\in H_{0}^{r}(\Omega),\,\|u+\varphi\|_{L^{2^{*r}}}=1}\|u\|_{r}^{2}
\label{eq1}
\end{equation}
and
\begin{equation}
S_{\theta,r}(\varphi)=\inf_{u\in H_{\theta}^{r}(\Omega),\,\|u+\varphi\|_{L^{2^{*r}}}=1}\|u\|_{r}^{2},
\label{eq2}
\end{equation}
where the function $\varphi\in L^{2^{*r}} (\Omega) \cap C(\Omega)$ and $2^{*r}=\frac{2N}{N-2 r}$ is the limiting Sobolev exponent in the imbedding $ H_0^r(\Omega) \hookrightarrow L^{q}(\Omega), $ $ 1 \leq q \leq 2^{*r}. $\\

The problem under consideration in this paper is related to the fact that \\ resembles some geometrical equations involving where lack of compactness occurs. \\
The statement of this problem on the bounded domain can be associated with \\ problems of the resolution of some minimization problem from geometry and physics, where the goal of our minimization problem is to determine the existence of a non-trivial minimum. \\
In 1986, Brezis considered in \cite{B} the first formulation of a problem which $ r=1$ and $ \varphi = 0$, see also \cite{BN1}, \cite{FS} and \cite{HY}. In \cite{BN2}, Bresiz and Nirenberg provided the first positive answer to this problem for $r=1$ stated in terms of an existence result from the infimum under the condition $ \varphi \not\equiv 0. $ Since then,
this problem has received many intention and this result has been improved in several ways. This include some results in the case where $ r=2$ and $ \varphi = 0$ (see for instance \cite{V1} and \cite{V2} ) that $S_{\theta,2}(0)= S_{0,2}(0)= S $ is the best Sobolev constant and $S$ is not achieved. In the papers \cite{BH} and \cite{GHP} the authors studied the problems (\ref{eq1}) and (\ref{eq2}) for the biharmonic operator $(-\Delta)^{2}$, see also \cite{GGS} for other study of biharmonic operator. In this paper, our motivation comes of the study of the critical growth of some polyharmonic operator. Polyharmonic equations have been considered in several works, see for exemple \cite{G}. This type of problems have many applications, we can cite the study of quantitate properties of solutions of semi-linear problems, the Paneitz type operator which appears in Willmore surfaces and in geometry, see \cite{P}.
In this paper we are interested in the two minimization problems (\ref{eq1}) and (\ref{eq2}) where the function $\varphi$ is given in $ L^{2^{*r}} (\Omega)\cap C(\Omega)$. Precisely, we consider the case $ r \geq 2$ and $ \varphi $ is not identically $0$ which is a natural generalization of the previous works. \\

Since we have $H_0^r(\Omega) \subset H_{\theta}^{r}(\Omega) $ we always have $ S_{\theta,r} \leq S_{0,r}$. A natural question arises is do we still have  $ S_{\theta,r}(\varphi) < S_{0,r}(\varphi) $ or $ S_{\theta,r}(\varphi) = S_{0,r}(\varphi) $ ? and in the case when the infima $S_{\theta,r}$ and $S_{0,r}$ are reached respectively by $u_\theta$ and $u_0$ can we know the sign of the Lagrange multiplier? and is $u_\theta $ in $H_0^r(\Omega)? $ \\

 Our main results can be stated as follows.
 \begin{theorem}~\\
Let $\Omega$ a regular bounded domain in $\R^{N}$ with $N\geq 2r+1$ and $\varphi\in L^{2^{*r}}\cap C(\Omega)\backslash \{0\}$. Then $S_{0,r}(\varphi)$ and $S_{\theta,r}(\varphi)$ are achieved.
\label{th1}
\end{theorem}

\begin{theorem}~\\
Let $\Omega$ a regular bounded domain in $\R^{N}$ with $N\geq 2r+1$ and $\varphi\in L^{2^{*r}}\cap C(\Omega)\backslash \{0\}$. We have
\begin{itemize}
\item[(i)] If $\|\varphi\|_{L^{2^{*r}}}<1$ and $\varphi$ has a constant sign on $\Omega$, then every minimizer of  $S_{\theta,r}(\varphi)$ is not in $H^{r}_{0}(\Omega)$ and we have $S_{\theta,r}(\varphi)< S_{0,r}(\varphi)$.
\item[(ii)] If $\varphi \in (H^{r}_{0}(\Omega))^{\perp}$, where $(H^{r}_{0}(\Omega))^{\perp}$ is the orthogonal space of $H^{r}_{0}(\Omega)$  in $H^{r}_{\theta}(\Omega)$, then every minimizer of $S_{\theta,r}(\varphi)$ is not in $H^{r}_{0}(\Omega)$ and we have $S_{\theta,r}(\varphi)< S_{0,r}(\varphi)$.
\item[(iii)] If $\|\varphi\|_{L^{2^{*r}}}>1$ and $\varphi \in H^{r}_{0}(\Omega)$  then $S_{\theta,r}(\varphi)=S_{0,r}(\varphi)$.
\end{itemize}
\label{th2}
\end{theorem}
\begin{remark}~\\
The proof of cases (i)-(ii) and (iii) are completely different, the last case is treated using the convexity of the problems, for more details see \cite{ET}.
\end{remark}

 The rest of paper is organized as follows: In section 2, we prove that the infimum in \eqref{eq1} and \eqref{eq2} are achieved where $\varphi\neq 0$ using some technical steps. In section 3, we present the proof of Theorem 2, more precisely we establish that $ S_{\theta,r}(\varphi) < S_{0,r}(\varphi), $ for $ \varphi$ satisfying suitable conditions.
 \section{Existence of minimizers}
In this section, we will prove Theorem \ref{th1}. This result is a natural generalization of the works \cite{BN2} and \cite{GHP}. \\
\textbf{Proof of Theorem \ref{th1}}.\\
We prove that $ S_{\theta,r} $ is achieved. The proof for $ S_{0,r} $ is similar. We follow an idea introduced in \cite{BN2} see also \cite{GHP}. \\
Let $\{u_{j}\}$ be a minimizing sequence for $S_{\theta,r}(\varphi)$, that is,
\begin{equation}
\|u_{j}+\varphi\|_{L^{2^{*r}}}=1
\label{m1}
\end{equation}
and
\begin{equation}
\|u_{j}\|_{r}^{2}=S_{\theta,r}(\varphi)+o(1).
\label{m2}
\end{equation}
An easy computations give that $\{u_{j}\}$ is bounded in $H_{\theta}^{r}(\Omega)$. Then, there exists s subsequence, still noted, $\{u_{j}\}$ such that
\begin{equation*}
\begin{array}{llll}
u_{j}\rightharpoonup u \quad\textrm{weakly in $H_{\theta}^{r}(\Omega)$},\\
u_{j}\rightarrow u \quad\textrm{strongly in $L^{t}(\Omega)$ for any $t< 2^{*r}$},\\
u_{j}\rightarrow u \quad\textrm{a.e on $\Omega$},\\
u_{j}\rightharpoonup u \quad \textrm{weakly in $L^{2^{*r}}(\Omega)$}.
\end{array}
\end{equation*}
Using the lower semi-continuity in (\ref{m1}) and (\ref{m2}), we obtain that
\begin{equation*}
\|u+\varphi\|_{L^{2^{*r}}}\leq 1,
\end{equation*}
and
\begin{equation}
\|u\|_{r}^{2}\leq S_{\theta,r}(\varphi).
\label{m3}
\end{equation}
In order to prove that $S_{\theta,r}(\varphi)$ is achieved by $u$, we need to establish $\|u+\varphi\|_{L^{2^{*r}}}= 1$. We proceed by contradiction, then we suppose that
\begin{equation}
\|u+\varphi\|_{L^{2^{*r}}}< 1 .
\label{contradiction1}
\end{equation}
We will prove the contradiction in four steps.
\begin{itemize}
\item{\textbf{Step 1}}~\\
We have
\begin{equation}
S_{\theta,r}(\varphi)-\|u\|_{r}^{2} \geq S_{r} \left[ 1-\int_{\Omega}|u+\varphi|^{2^{*r}}\right]^{\frac{2}{2^{*r}}}.
\label{eqq2}
\end{equation}
Indeed, let $v_{j}=u_{j}-u$. We have
\begin{equation*}
v_{j}\rightharpoonup 0 \quad \textrm{weakly in $H^{r}_{\theta}(\Omega)$}.
\end{equation*}
\begin{equation*}
v_{j}\rightarrow 0 \quad \textrm{a.e on $\Omega$}.
\end{equation*}
Looking in the definition of $S_{r}$ we write
\begin{equation}
\displaystyle \|v_{j}\|_{r}^{2}\geq S_{r}\|v_{j}\|_{L^{2^{*r}}}^{2}.
\label{prs11}
\end{equation}
From (\ref{m1}), we see that
\begin{equation*}
1=\|u+\varphi\|_{L^{2^{*r}}}^{2^{*r}}+\|v_{j}\|_{L^{2^{*r}}}^{2^{*r}}+o(1),
\end{equation*}
By Brezis-Lieb Lemma \cite{BL}, we have
\begin{equation}
\|v_{j}\|_{L^{2^{*r}}}^{2}=\left[1-\|u+\varphi\|_{L^{2^{*r}}}^{2^{*r}}\right]^{\frac{2}{2^{*r}}}+o(1),
\label{prs12}
\end{equation}
Inserting (\ref{prs12}) into (\ref{prs11}), we get
\begin{equation}
\displaystyle \|v_{j}\|_{r}^{2}\geq S_{r}\left[1-\|u+\varphi\|_{L^{2^{*r}}}^{2^{*r}}\right]^{\frac{2}{2^{*r}}}.
\label{prs13}
\end{equation}
On the other hand, from (\ref{m2}), we write
\begin{equation}
\|v_{j}\|_{r}^{2}=S_{\theta,r}(\varphi)-\|u\|_{r}^{2}+o(1).
\label{prs14}
\end{equation}
Inserting (\ref{prs13}) into (\ref{prs14}) we obtain (\ref{eqq2}).
\item{\textbf{Step 2}}~\\
Let $v \in H^{r}_{\theta}$ such that $\|v+ \varphi\|_{L^{2^{*r}}}\leq 1$, we have
\begin{equation}
S_{\theta,r}(\varphi)-\|v\|_{r}^{2} \leq S_{r} \left[ 1-\|v+ \varphi\|_{L^{2^{*r}}}^{2^{*r}}\right]^{\frac{2}{2^{*r}}},
\label{eq3}
\end{equation}
and thus
\begin{equation}
S_{\theta,r}(\varphi)-\|u\|_{r}^{2} = S_{r} \left[ 1-\|u+ \varphi\|_{L^{2^{*r}}}^{2^{*r}}\right]^{\frac{2}{2^{*r}}}.
\label{eq4}
\end{equation}
Indeed, let $v\in H^{r}_{\theta}(\Omega)$ such that $\|v+\varphi\|_{L^{2^{*r}}}\leq 1$. Suppose that $\|v+\varphi\|_{L^{2^{*r}}}< 1$, otherwise (\ref{eq3}) comes directly from the definition of $S_{\theta,r}$. There exists $c_{\epsilon}> 0$ such that $\|v+\varphi+c_{\varepsilon} u_{x_{0},\varepsilon}\|_{L^{2^{*r}}}=1 $ where $u_{x_{0},\varepsilon}$ is an extremal function associate to the best Sobolev constant $S_{r}$ defined by
\begin{equation} \label{D}
u_{x_{0},\varepsilon}(x)=\frac{\varepsilon^{\frac{N-2r}{2}} \xi}{(\varepsilon^{2}+|x-x_{0}|^{2})^{\frac{N-2 r}{2}}}.
\end{equation}
where $x_{0}\in \Omega$ and $\xi \in C_{0}^{\infty}(B(x_{0},\,R))$ be a fixed cut-off function satisfying $0\leq\xi\leq 1$ and $\xi\equiv 1$ on $B(x_{0},\,\frac{R}{2})$ with $R$ a positive constant.\\
We have from \cite{S},
\begin{equation*}
(-\Delta)^{j}u_{x_{0},\varepsilon}(t)=\frac{\varepsilon^{\frac{N-2r+4j}{2}}\sum_{i=0}^{j}G(i,j)t^{2 i}}{(\varepsilon^{2}+t^{2})^{\frac{N-2r+4 j}{2}}},\quad\textrm{for $j=1,2,\ldots, r$},
\end{equation*}
where
\begin{equation*}
G(i,j)=2^{i}(^{j}_{i}) K_{j} D(i, j) E(i, j),
\end{equation*}
with
\begin{equation*}
K_{j}=\displaystyle\Pi_{h=0}^{j-1}(N-2r+2h),
\end{equation*}
\begin{equation*}
D(i,j)=\left\{\begin{array}{lll}
1\quad \textrm{if $i=0$}\\[\medskipamount]
\displaystyle\Pi_{h=0}^{j-1}(r-h)\quad \textrm{if $i=1,2,\ldots, j$}
\end{array}
\right.
\end{equation*}
and
\begin{equation*}
E(i,j)=\left\{\begin{array}{lll}
\displaystyle\Pi_{h=0}^{j-1}(N+2h)\quad \textrm{if $i=0,1,\ldots, j-1$}\\[\medskipamount]
1 \quad\textrm{if $i=j$}\\[\medskipamount]
0\quad \textrm{if $i\geq j+1$}.
\end{array}
\right.
\end{equation*}

From \cite{EFJ}, we have

\begin{equation}
\|u_{x_{0},\varepsilon}\|_{L^{2^{*r}}}^{2}=\frac{K}{S_{r}}+O(\varepsilon^{N-2r}),
\label{eqs21}
\end{equation}
\begin{equation}
\| u_{x_{0},\varepsilon}\|_{r}^{2}=K+O(\varepsilon^{N-2r}),
\label{eqs22}
\end{equation}
\begin{equation}
u_{x_{0},\varepsilon}\rightharpoonup 0 \quad\textrm{in $H^{r}(\Omega)$},
\label{eqs23}
\end{equation}
where $K$ is a positive constant.\\
Now, we have $\|v+\varphi+c_{\varepsilon} u_{x_{0},\varepsilon}\|_{L^{2^{*r}}}=1$. Using Bresiz-Lieb Lemma, we write
\begin{equation*}
c_{\varepsilon}^{2^{*r}}\|u_{x_{0},\varepsilon}\|_{L^{2^{*r}}}^{2^{*r}}=1-\|v+\varphi\|_{L^{2^{*r}}}^{2^{*r}}+o(1),
\end{equation*}
Therefore
\begin{equation}
c_{\varepsilon}^{2}=\frac{S_{r}}{K}\left[1-\|v+\varphi\|_{L^{2^{*r}}}^{2^{*r}}\right]^{\frac{2}{2^{*r}}}+o(1).
\label{eqs24}
\end{equation}
On the other hand, we have\\
\textbf{If $r$ is even}
\begin{equation}
\begin{array}{llll}
S_{\theta,r}(\varphi)&\leq& \|v+c_{\varepsilon} u_{x_{0},\varepsilon}\|_{r}^{2}\\[\medskipamount]
&\leq& \|v\|_{r}^{2}+c_{\varepsilon}^{2}\|u_{x_{0},\varepsilon}\|_{r}^{2}+2 c_{\varepsilon}\displaystyle{\int_{\Omega}(\Delta)^{\frac{r}{2}} v (\Delta)^{\frac{r}{2}} u_{x_{0},\varepsilon}dx}\\[\medskipamount]
&\leq& \|v\|_{r}^{2}+c_{\varepsilon}^{2} K+2 \varepsilon^{\frac{N}{2}}c_{\varepsilon} \displaystyle{\int_{\Omega}(\Delta)^{\frac{r}{2}} v \frac{\sum_{i=0}^{\frac{r}{2}}G(i,\frac{r}{2})t^{2 i}}{(\varepsilon^{2}+t^{2})^{\frac{N}{2}}}dx}+o(\varepsilon^{\frac{N-2r}{2}}).
\end{array}
\label{eqs25}
\end{equation}
\textbf{If $r$ is odd}
\begin{equation}
\begin{array}{llll}
S_{\theta,r}(\varphi)&\leq& \|v+c_{\varepsilon} u_{x_{0},\varepsilon}\|_{r}^{2}\\[\medskipamount]
&\leq& \|v\|_{r}^{2}+c_{\varepsilon}^{2}\|u_{x_{0},\varepsilon}\|_{r}^{2}+2 c_{\varepsilon}\displaystyle{\int_{\Omega} \vert \nabla(-\Delta)^{\frac{r-1}{2}} v \vert \vert \nabla(-\Delta)^{\frac{r-1}{2}} u_{x_{0},\varepsilon} \vert dx}\\[\medskipamount]
&\leq& \|v\|_{r}^{2}+c_{\varepsilon}^{2} K+2 \varepsilon^{\frac{N-2r +4j}{2}}c_{\varepsilon} \displaystyle{ \int_{\Omega} \vert \nabla(-\Delta)^{\frac{r-1}{2}} v \vert \left[ \frac{\sum_{i=0}^{j}G(i,j)t^{2 i}}{(\varepsilon^{2}+t^{2})^{\frac{N-2}{2}}} - \frac{t^2}{2} \frac{(N-2)}{(\varepsilon^{2}+t^{2})^{N-2}}\right] dx}\\
&+&o(\varepsilon^{\frac{N-2r}{2}}).
\end{array}
\label{eqs26}
\end{equation}
In the two cases of $r$, using (\ref{eqs22}) and (\ref{eqs24}), the inequality (\ref{eqs25}) or (\ref{eqs26}) becomes
\begin{equation*}
S_{\theta,r}(\varphi)\leq \|v\|_{r}^{2}+S_{r}\left[1-\|v + \varphi\|_{L^{2^{*r}}}^{2^{*r}}\right]^{\frac{2}{2^{*r}}}+o(\varepsilon^{\frac{N-2r}{2}}).
\end{equation*}
Therefore we deduce (\ref{eq3}). Also, replace $v$ par $u$ in (\ref{eq3}) and using step 1 we get (\ref{eq4}).
\item{\textbf{Step 3}}~\\
According to assumption $(\ref{contradiction1})$: \\
If $r$ is even then
\begin{equation}
\displaystyle{\int_{\Omega}(\Delta)^{r/2}u(\Delta)^{r/2} v dx} =
S_{r} \left[ 1-\|u+\varphi\|_{L^{2^{*r}}}^{2^{*r}}\right]^{\frac{2}{2^{*r}}-1}\displaystyle{\int_{\Omega}|u+\varphi|^{2^{*r}-2}(u+\varphi)v dx}
\label{eq5}
\end{equation}
for every $v\in H^{r}_{\theta}(\Omega)$.\\
If $r$ is odd then
\begin{equation}
\displaystyle{ \int_{\Omega}\nabla(-\Delta)^{\frac{r-1}{2}}u \, \nabla(-\Delta)^{\frac{r-1}{2}} v dx }= S_{r} \left[ 1-\|u+ \varphi\|_{L^{2^{*r}}}^{2^{*r}}\right]^{\frac{2}{2^{*r}}-1} \displaystyle{ \int_{\Omega}|u+\varphi|^{2^{*r}-2}(u+\varphi) v dx}
\label{eq*}
\end{equation}
for every $v\in H^{r}_{\theta}(\Omega)$.\\
Indeed, let $v \in H^{r}_{\theta}(\Omega)$. Since $\|u+\varphi\|_{L^{2^{*r}}}< 1$, there exists $t_{0}> 0$ such that for all $|t|< t_{0}$ we have
\begin{equation*}
\|u+\varphi+ t v\|_{L^{2^{*r}}}<1.
\end{equation*}
Therefore, from Step 2, we have
\begin{equation*}
S_{\theta,r}(\varphi)-\|u+t v\|_{r}^{2} \leq S_{r} \left[ 1-\|u+ t v+ \varphi\|_{L^{2^{*r}}}^{2^{*r}}\right]^{\frac{2}{2^{*r}}}.
\end{equation*}
At this stage, we distinguish two cases:\\
\textbf{If $r$ is even } then
\begin{equation*}
S_{\theta,r}(\varphi)-\|u\|_{r}^{2}-2 t \displaystyle{\int_{\Omega}(\Delta)^{\frac{r}{2}} u (\Delta)^{\frac{r}{2}} vdx+o(t)\leq S_{r}
\left[ 1-\|u+ t v+ \varphi\|_{L^{2^{*r}}}^{2^{*r}}\right]^{\frac{2}{2^{*r}}}},
\end{equation*}
some computations give
\begin{equation*}
\begin{array}{lll}
S_{\theta,r}(\varphi)-\|u\|_{r}^{2}-2 t \displaystyle{\int_{\Omega}(\Delta)^{\frac{r}{2}} u (\Delta)^{\frac{r}{2}} vdx+o(t)\leq S_{r}
\left[ 1-\|u+ \varphi\|_{L^{2^{*r}}}^{2^{*r}}\right]^{\frac{2}{2^{*r}}}}\\[\medskipamount]
\displaystyle\times \left[1-2 t (1-\|u+ \varphi\|_{L^{2^{*r}}}^{2^{*r}})^{-1}\int_{\Omega}|u+\varphi|^{2^{*r}-2}(u+\varphi) v dx+o(t) \right].
\end{array}
\end{equation*}
Using (\ref{eq4}), we obtain
\begin{equation*}
\begin{array}{lll}
-2 t \displaystyle{\int_{\Omega}(\Delta)^{\frac{r}{2}} u (\Delta)^{\frac{r}{2}} v dx}+o(t)&\leq& -2 t S_{r} (1-\|u+ \varphi\|_{L^{2^{*r}}}^{2^{*r}})^{\frac{2}{2^{*r}}-1}\displaystyle\int_{\Omega}|u+\varphi|^{2^{*r}-2}(u+\varphi) v dx\\[\medskipamount]
&+&o(t).
\end{array}
\label{s31}
\end{equation*}
We deduce (\ref{eq5}) by letting $t$ goes to $0^{\pm}$.\\
\textbf{If $r$ is odd} then\\
Using again Step 2, we have
\begin{equation*}
S_{\theta,r}(\varphi)-\|u\|_{r}^{2}-2 t \displaystyle{ \int_{\Omega}\nabla(-\Delta)^{\frac{r-1}{2}} u \nabla(-\Delta)^{\frac{r-1}{2}} vdx}+o(t)\leq S_{r}
\left[ 1-\|u+ t v+ \varphi\|_{L^{2^{*r}}}^{2^{*r}}\right]^{\frac{2}{2^{*r}}},
\end{equation*}
some computations give
\begin{equation*}
\begin{array}{lll}
S_{\theta,r}(\varphi)-\|u\|_{r}^{2}-2 t \displaystyle{ \int_{\Omega}\nabla(-\Delta)^{\frac{r-1}{2}} u \nabla(-\Delta)^{\frac{r-1}{2}} vdx}+o(t)\leq S_{r}
\left[ 1-\|u+ \varphi\|_{L^{2^{*r}}}^{2^{*r}}\right]^{\frac{2}{2^{*r}}}\\[\medskipamount]
\times \left[1-2 t (1-\|u+ \varphi\|_{L^{2^{*r}}}^{2^{*r}})^{-1} \displaystyle{\int_{\Omega}|u+\varphi|^{2^{*r}-2}(u+\varphi) v dx}+o(t) \right].
\end{array}
\end{equation*}
Using (\ref{eq4}), we obtain
\begin{eqnarray*}
&-2 t \displaystyle{\int_{\Omega}\nabla(-\Delta)^{\frac{r-1}{2}} u \nabla(-\Delta)^{\frac{r-1}{2}} v dx}+o(t) \\
&\leq -2 t S_{r} (1-\|u+ \varphi\|_{L^{2^{*r}}}^{2^{*r}})^{\frac{2}{2^{*r}}-1}\displaystyle{\int_{\Omega}|u+\varphi|^{2^{*r}-2}(u+\varphi) v dx}+o(t).
\label{s32}
\end{eqnarray*}
We get (\ref{eq*}) by letting $t$ goes to $0^{\pm}$.\\
Now, we will show that the hypothesis $(\ref{contradiction1})$ is not true and leads to a contradiction with (\ref{eq4}).
\item{\textbf{Step 4}}~\\
The assumption $(\ref{contradiction1})$ implies that
\begin{equation}
S_{\theta,r}(\varphi)-\|u\|_{r}^{2} < S_{r}\displaystyle{ \left[ 1- \int_{\Omega}|u+\varphi|^{2^{*r}}\right]^{\frac{2}{2^{*r}}}}.
\label{eq6}
\end{equation}
Indeed, we have that $u+\varphi \not\equiv 0$, otherwise, from (\ref{eq5}) we obtain $\|u \|_{r}=0$ therefore $u=0$ and $\varphi=0$ which is false. Since we may replace $u$ by $-u$ and $\varphi$ by $-\varphi$, we may assume, without loss of generality, that $u+\varphi>0$ in a set $\Sigma$ of a positive measure in a ball $B(x_{0},\,\frac{R}{2})\subset \Omega$ with $R$ a positive constant. Then, let $x_{0} \in \Sigma$ such that $(u+\varphi)(x_{0})> 0$. \\
As in the proof of Step 2 there exists $c_{\varepsilon}>0$ such that $\displaystyle \|u+\varphi+c_{\varepsilon} u_{x_{0},\varepsilon}\|_{L^{2^{*r}}}=1$, where $c_{\varepsilon}$ is defined in (\ref{eqs24}).\\
We use $u+c_{\varepsilon} u_{x_{0},\varepsilon}$ as a testing function of \eqref{pb2} gives that\\
\textbf{If $r$ is even}
\begin{equation}
S_{\theta,r}(\varphi)\leq \|u\|_{r}^{2}+c_{\varepsilon}^{2}\displaystyle{\int_{\Omega} ((\Delta)^{\frac{r}{2}}u_{x_{0},\varepsilon})^{2}dx}+2 c_{\varepsilon} \displaystyle{\int_{\Omega} (\Delta)^{\frac{r}{2}} u_{x_{0},\varepsilon} (\Delta)^{\frac{r}{2}} u\, dx}.
\label{imen1}
\end{equation}
Let $\delta_{\varepsilon}$ and $c_{0}$ be given by
\begin{equation}
c_{\varepsilon}=c_{0}(1-\delta_{\varepsilon}),\qquad c_{0}^{2}=\frac{S_{r}}{K}[1-\|u+\varphi\|_{L^{2^{*r}}}^{{2^{*r}}}]^{\frac{2}{2^{*r}}}.
\label{eqmercredi}
\end{equation}
Therefore
\begin{equation}
[1-\|u+\varphi\|_{L^{2^{*r}}}^{L^{2^{*r}}}]^{-1}=c_{0}^{-2^{*r}}(\frac{S_{r}}{K})^{\frac{2^{*r}}{2}}.
\label{eqmercredi1}
\end{equation}
Using (\ref{eqs22}) and applying Step 3 with $v=u_{x_{0},\varepsilon}$, (\ref{imen1})
\begin{equation*}
\begin{array}{lll}
S_{\theta,r}(\varphi)- \|u\|_{r}^{2}&\leq& c_{0}^{2}(1-\delta_{\varepsilon})^{2}(K+O(\varepsilon^{N-2 r}))\\[\medskipamount]
&+&2 c_{\varepsilon}S_{r} \left[ 1-\|u+\varphi\|_{L^{2^{*r}}}^{2^{*r}}\right]^{\frac{2}{2^{*r}}-1}\displaystyle{\int_{\Omega}|u+\varphi|^{2^{*r}-2}(u+\varphi)u_{x_{0},\varepsilon} dx}.
\end{array}
\end{equation*}
Using (\ref{eqs24}) we obtain
\begin{equation}
\begin{array}{lll}
S_{\theta,r}(\varphi)- \|u\|_{r}^{2}&\leq& c_{0}^{2}(1-\delta_{\varepsilon})^{2}(K+O(\varepsilon^{N-2r}))\\[\medskipamount]
&+&2 c_{\varepsilon}c_{0}^{2}K \left[ 1-\|u+\varphi\|_{L^{2^{*r}}}^{2^{*r}}\right]^{-1}\displaystyle{\int_{\Omega}|u+\varphi|^{2^{*r}-2}(u+\varphi)u_{x_{0},\varepsilon} dx}.
\end{array}
\label{eqmercredi2}
\end{equation}
Using (\ref{eqmercredi}) we write
\begin{equation}
\begin{array}{lll}
S_{\theta,r}(\varphi)- \|u\|_{r}^{2}&\leq& S_{r}[1-\|u+\varphi\|_{L^{2^{*r}}}^{{2^{*r}}}]^{\frac{2}{2^{*r}}}(1-\delta_{\varepsilon})^{2}(1+O(\varepsilon^{N-2r}))\\[\medskipamount]
&+&2 c_{\varepsilon}c_{0}^{2}K \left[ 1-\|u+\varphi\|_{L^{2^{*r}}}^{2^{*r}}\right]^{-1}\displaystyle{\int_{\Omega}|u+\varphi|^{2^{*r}-2}(u+\varphi)u_{x_{0},\varepsilon} dx.}
\end{array}
\label{eqs41}
\end{equation}
We distinguish two cases:\\
\textbf{If $2^{*r}\geq 3$} we apply the following inequality
$$(x+y)^{p}-x^{p}-y^{p}-px^{p-1}y-px y^{p-1}\geq 0,\quad x,y\geq 0,\quad p\geq 3.$$
For $x=u+\varphi$ and $y=c_{\varepsilon} u_{x_{0},\varepsilon}$, using (\ref{eqs21}) and (\ref{eqmercredi}) we write
\begin{equation*}
\begin{array}{llll}
c_{\varepsilon}\displaystyle{ \int_{\Omega}|u+\varphi|^{2^{*r}-2}(u+\varphi) u_{x_{0},\varepsilon}dx} &\leq& \frac{1}{2^{*r}}[1-\|u+\varphi\|_{L^{2^{*r}}}^{2^{*r}}-c_{\varepsilon}^{2^{*r}}\|u_{x_{0},\varepsilon}\|_{2^{*r}}^{2^{*r}}] \\[\medskipamount]
&-&\displaystyle{c_{\varepsilon}^{2^{*r}-1}\int_{\Omega}|u_{x_{0},\varepsilon}|^{2^{*r}-1}(u+\varphi) dx}\\[\medskipamount]
&\leq&\frac{1}{2^{*r}} c_{0}^{2^{*r}}(\frac{K}{S_{r}})^{\frac{2^{*r}}{2}}\\[\medskipamount]
&-&\frac{1}{2^{*r}}c_{0}^{2^{*r}}(1-\delta_{\varepsilon})^{2^{*r}}
\left((\frac{K}{S_r})^{\frac{2^{*r}}{2}}+ O(\varepsilon^{N})\right)\\[\medskipamount]
&-&\displaystyle c_{\varepsilon}^{2^{*r}-1}\int_{\Omega}|u_{x_{0},\varepsilon}|^{2^{*r}-1}(u+\varphi) dx\\[\medskipamount]
\end{array}
\end{equation*}
On the other hand, a easy computation gives
\begin{equation}
\displaystyle{\int_{\Omega}|u_{x_{0},\varepsilon}|^{2^{*r}-1}(u+\varphi)dx=D \varepsilon^{\frac{N-2r}{2}}(u+\varphi)(x_{0})+o(\varepsilon^{\frac{N-2r}{2}})}
\label{eqs42}
\end{equation}
where $D=\displaystyle \int_{\R^{N}}\frac{1}{(1+|y|^{2})^{\frac{N+2r}{2}}}$.\\
Then
\begin{equation}
\begin{array}{lll}
c_{\varepsilon}\displaystyle{ \int_{\Omega}|u+\varphi|^{2^{*r}-2}(u+\varphi) u_{x_{0},\varepsilon}dx} &\leq& c_{0}^{2^{*r}}(\frac{K}{S_{r}})^{\frac{2^{*r}}{2}}(\delta_{\varepsilon}- \frac{1}{2}(2^{*r}-1)\delta_{\varepsilon}^{2}
+o(\delta_{\varepsilon}^{2})+o(\varepsilon^{N}))\\[\medskipamount]
&-& \frac{1}{2^{*r}} c_{\varepsilon}^{2^{*r}-1}D\varepsilon^{\frac{N-2r}{2}}(u+\varphi)(x_0)+o(\varepsilon^{\frac{N-2r}{2}}).
\end{array}
\label{eqsamedi1}
\end{equation}
Inserting (\ref{eqsamedi1}) into (\ref{eqs41}), we get
\begin{equation*}
\begin{array}{lll}
S_{\theta,r}(\varphi)- \|u\|_{r}^{2}&\leq& c_{0}^{2}K (1-2\delta_{\varepsilon}+\delta_{\varepsilon}^{2})+O(\varepsilon^{N-2r})\\[\medskipamount]
&+&2c_{0}^{2} (\frac{K}{S_{r}})^{\frac{2^{*r}}{2}} K (1-\|u+\varphi\|_{L^{2^{*r}}}^{2^{*r}})^{-1} \left(\delta_{\varepsilon}-\frac{1}{2}(2^{*r}-1)\delta_{\varepsilon}^{2}+o(\delta_{\varepsilon}^{2})\right)+o(\varepsilon^{N})\\[\medskipamount]
&-& 2\, c_{\varepsilon}^{2^{*r}-1} c_{0}^{2} K(1-\|u+\varphi\|_{L^{2^{*r}}}^{2^{*r}})^{-1}D\varepsilon^{\frac{N-2r}{2}}(u+\varphi)(x_0)+o(\varepsilon^{\frac{N-2r}{2}}).
\end{array}
\end{equation*}
Then
\begin{equation*}
\begin{array}{lll}
S_{\theta,r}(\varphi)- \|u\|_{r}^{2}&\leq& c_{0}^{2}K-2c_{0}^{2}K \delta_{\varepsilon}+ c_{0}^{2}K \delta_{\varepsilon}^{2}+ o(\varepsilon^{\frac{N-2r}{2}})\\
[\medskipamount]
&+& 2c_{0}^{2^{*r}+2}(\frac{K}{S_{r}})^{\frac{2^{*r}}{2}} K (1-\|u+\varphi\|_{L^{2^{*r}}}^{2^{*r}})^{-1} \delta_{\varepsilon}\\[\medskipamount]
&-& (2^{*r}-1)c_{0}^{2^{*r}+2}(\frac{K}{S_{r}})^{\frac{2^{*r}}{2}} K (1-\|u+\varphi\|_{L^{2^{*r}}}^{2^{*r}})^{-1}  \delta_{\varepsilon}^{2}+o(\delta_{\varepsilon}^{2})\\
[\medskipamount]
&-& 2\, c_{\varepsilon}^{2^{*r}-1} c_{0}^{2}K(1-\|u+\varphi\|_{L^{2^{*r}}}^{{2^{*r}}})^{-1}D\varepsilon^{\frac{N-2r}{2}}(u+\varphi)(x_0)+o(\varepsilon^{\frac{N-2r}{2}}).
\end{array}
\end{equation*}
Using (\ref{eqmercredi1}), we get
\begin{equation*}
\begin{array}{lll}
S_{\theta,r}(\varphi)- \|u\|_{r}^{2}&\leq& c_{0}^{2}K-c_{0}^{2}K (2^{2^{*r}}-2)\delta_{\varepsilon}^{2}+o(\delta_{\varepsilon}^{2})\\
[\medskipamount]
&-& 2\, c_{\varepsilon} K(\frac{S_{r}}{K})^{\frac{2^{*r}}{2}} D\varepsilon^{\frac{N-2r}{2}}(u+\varphi)(x_0)+o(\varepsilon^{\frac{N-2r}{2}}).
\end{array}
\end{equation*}
Consequently
\begin{equation*}
S_{\theta,r}(\varphi)- \|u\|_{r}^{2}<c_{0}^{2}K=S_{r}\left[1-\|\|u+\varphi\|_{L^{2^{*r}}}^{L^{2^{*r}}} \right]^{\frac{2}{2^{*r}}}.
\end{equation*}
\textbf{If $2^{*r}\leq 3$:}\\
We use the following inequality see [\cite{BN2}, Lemma 4] and \cite{GHP},
\begin{equation}
\left||x+y|^{p}-|x|^{p}-|y|^{p}-p x y (|x|^{p-2}+|y|^{p-2}) \right|\leq \left\{\begin{array}{lll} C |x|^{p-1}|y|\quad \textrm{if $|x|\leq |y|$},\\[\smallskipamount]
C |x||y|^{p-1} \quad \textrm{if $|x|\geq |y|$},
\end{array}
\right.
\label{eqmercredi3}
\end{equation}
for $x,\,y \in \R$, where $C=C(p)$ a positive a constant.\\
Define
\begin{equation}
\begin{array}{lll}
A_{\varepsilon}&:=&\displaystyle{ 1-\int_{\Omega}|u+\varphi|^{2^{*r}}dx-c_{\varepsilon}^{2^{*r}}\int_{\Omega}|u_{x_{0},\varepsilon}|^{2^{*r}}dx-2^{*r} c_{\varepsilon}^{2^{*r}-1}\int_{\Omega}|u_{x_{0},\varepsilon}|^{2^{*r}-1}(u+\varphi)dx}\\[\medskipamount]
&-&\displaystyle 2^{*r}c_{\varepsilon} \int_{\Omega}|u+\varphi|^{2^{*r}-2}(u+\varphi)u_{x_{0},\varepsilon}dx.
\end{array}
\label{samedi1}
\end{equation}
Applying the inequality (\ref{eqmercredi3}) with $x=u+\varphi$, $y= c_{\varepsilon} u_{x_{0},\varepsilon}$ and we suppose that $x_0 = 0$ for simplicity, we write
\begin{equation*}
\begin{array}{lll}
|A_{\varepsilon}|&\leq&\displaystyle C \left\{c_{\varepsilon} \int_{\left\lbrace |u+\varphi|\leq c_{\varepsilon } u_{x_{0},\varepsilon} \right\rbrace}|u+\varphi |^{2^{*r}-1}u_{x_{0},\varepsilon}dx+c_{\varepsilon}^{2^{*r}-1}\int_{\left\lbrace |u+\varepsilon|\geq c_{\varepsilon} u_{x_{0},\varepsilon}\right\rbrace }|u+\varphi|u_{x_{0},\varepsilon}^{2^{*r}-1}dx\right\}\\[\bigskipamount]
|A_{\varepsilon}|&\leq& A_{\varepsilon}^{1}+A_{\varepsilon}^{2}.
\end{array}
\end{equation*}
On one hand, we have
\begin{equation*}
A_{\varepsilon}^{1}\leq \displaystyle C_{1} \varepsilon^{\frac{N-2r}{2}}\int_{0}^{c_{1}\varepsilon^{\frac{1}{2}}}\frac{z^{N-1}}{(\varepsilon^{2}+z^{2})^{\frac{N-2r}{2}}}dz\quad\textrm{since $\displaystyle \{|u+\varphi|\leq c_{\varepsilon} u_{x_{0},\varepsilon}\}\subset \{|x|\leq c_{1}\varepsilon^{\frac{1}{2}}\}$,}
\end{equation*}

and
\begin{equation*}
A_{\varepsilon}^{2}\leq \displaystyle C_{2} \varepsilon^{\frac{N+2r}{2}}\int_{c_{2}\varepsilon^{\frac{1}{2}}}^{c_{3}}\underline{}\frac{z^{N-1}}{(\varepsilon^{2}+z^{2})^{\frac{N+2r}{2}}}dz\quad\textrm{since $\displaystyle \{|u+\varphi|\geq c_{\varepsilon} u_{x_{0},\varepsilon}\}\subset \{c_{2}\varepsilon^{\frac{1}{2}} \leq |x|\leq c_{3}\}$,}
\end{equation*}
where $c_{1}$, $c_{2}$ and $c_{3}$ are some positive constants.\\
On the other hand, some computations give
\begin{equation*}
\begin{array}{lll}
 \displaystyle{\int_{0}^{c_{1}\varepsilon^{\frac{1}{2}}}\frac{z^{N-1}}{(\varepsilon^{2}+z^{2})^{\frac{N-2r}{2}}}dz}\leq
 \displaystyle{\int_{0}^{c_{1}\varepsilon^{\frac{1}{2}}}z^{2r-1}dz= \frac{1}{2r}c_1 \varepsilon^{r}=O(\varepsilon^{r})},

 \end{array}
\end{equation*}
and
\begin{equation*}
\displaystyle{\int_{c_{2}\varepsilon^{\frac{1}{2}}}^{c_{3}}\frac{z^{N-1}}{(\varepsilon^{2}+z^{2})^{\frac{N+2r}{2}}}dz}\leq \int_{c_{2}\varepsilon^{\frac{1}{2}}}^{c_{3}}z^{-2r-1}dz=-\frac{1}{2r}[c_3^{-2r}-(c_{2})^{-2r}\varepsilon^{-r}] =\varepsilon^{-r}( K_2-K_3 \varepsilon^{r} ) = O(\varepsilon^{-r}).
\end{equation*}
Therefore $\displaystyle A_{\varepsilon}^{1}=O(\varepsilon^{\frac{N}{2}})=o(\varepsilon^{\frac{N-2r}{2}})$ and $A_{\varepsilon}^{2}=O(\varepsilon^{\frac{N}{2}})=o(\varepsilon^{\frac{N-2r}{2}})$.
Thus
\begin{equation}
\displaystyle A_{\varepsilon}=o(\varepsilon^{\frac{N-2r}{2}}).
\label{eqs43}
\end{equation}
Combining (\ref{eqs21}), (\ref{samedi1}) and (\ref{eqs43}) we get
\begin{equation*}
\begin{array}{lll}
c_{\varepsilon}\displaystyle{ \int_{\Omega}|u+\varphi|^{{2^{*r}-2}}(u+\varphi)u_{x_{0},\varepsilon}dx}&=&\frac{1}{2^{*r}}\left[1-\|u+\varphi\|_{L^{2^{*r}}}^{2^{*r}}-c_{0}^{2^{*r}}(1-{2^{*r}}\delta_{\varepsilon})
\left(\frac{K}{S_{r}}\right)^{\frac{2^{*r}}{2}}+o(\varepsilon^{N-2r}) \right]\\[\medskipamount]
&-&\displaystyle{ c_{\varepsilon}^{2^{*r}-1}\displaystyle{\int_{\Omega}|u_{x_{0},\varepsilon}|^{2^{*r}-1}(u+\varphi)dx}+o(\delta_{\varepsilon})+o(\varepsilon^{\frac{N-2r}{2}})}.
\end{array}
\end{equation*}
Using (\ref{eqmercredi1}), an easy computation gives
\begin{equation}
\displaystyle{ c_{\varepsilon}\int_{\Omega}|u+\varphi|^{2^{*r}-2}(u+\varphi)u_{x_{0},\varepsilon}dx}=\delta_{\varepsilon}c_{0}^{2^{*r}}(\frac{K}{S_{r}})^{\frac{2^{*r}}{2}}- c_{0}^{2^{*r}-1}\displaystyle{ \int_{\Omega}|u_{x_{0},\varepsilon}|^{2^{*r}-1}(u+\varphi)dx} +o(\delta_{\varepsilon})+o(\varepsilon^{\frac{N-2r}{2}}).
\label{eqmercredi4}
\end{equation}
On the other way, we have
\begin{equation}
{c_{\varepsilon}\displaystyle{ \int_{\Omega}|u+\varphi|^{2^{*r}-2}(u+\varphi)u_{x_{0},\varepsilon}dx}= c_{\varepsilon} \varepsilon^{\frac{N-2r}{2}} \int_{\Omega} \frac{(u+\varphi)^{2^{*r}-1}(x) dx}{\vert x \vert^{N-2r} }+o(\varepsilon^{\frac{N-2r}{2}})= O(\varepsilon^{\frac{N-2r}{2}}) }
\label{eqs45}
\end{equation}
Putting (\ref{eqs42}) and (\ref{eqs45}) into (\ref{eqmercredi4}) we deduce
\begin{equation}
{\delta_{\varepsilon}=O(\varepsilon^{\frac{N-2r}{2}}).}
\label{eqs46}
\end{equation}
Now, returning to (\ref{eqmercredi2}) and using (\ref{eqmercredi4}), we write
\begin{equation*}
\begin{array}{lll}
S_{\theta,r}(\varphi)- \|u\|_{r}^{2}&\leq& c_{0}^{2}K-2\delta_{\varepsilon}c_{0}^{2}K+2 c_{0}^{2}K(1-\|u+\varphi\|_{2^{*r}}^{2^{*r}})^{-1}\\[\medskipamount]
&\times&\displaystyle\displaystyle \left[\delta_{\varepsilon}c_{0}^{2^{*r}}\left(\frac{K}{S_{r}}\right)^{\frac{2^{*r}}{2}}-c_{0}^{2^{*r}-1}\int_{\Omega}|u_{x_{0},\varepsilon}|^{2^{*r}-1}(u+\varphi)dx\right]
+o(\delta_{\varepsilon})+o(\varepsilon^{\frac{N-2r}{2}}).
\end{array}
\end{equation*}
From (\ref{eqs46}), we get
\begin{equation}
S_{\theta,r}(\varphi)- \|u\|_{r}^{2}\leq c_{0}^{2}K-2c_{0}K\left(\frac{S_{r}}{K}\right)^{\frac{2^{*r}}{2}}\displaystyle{\int_{\Omega}|u_{x_{0},\varepsilon}|^{2^{*r}-1}(u+\varphi)+o(\varepsilon^{\frac{N-2r}{2}})}.
\label{eqs47}
\end{equation}
Using (\ref{eqs42}) we write
\begin{equation*}
S_{\theta,r}(\varphi)- \|u\|_{r}^{2}\leq c_{0}^{2}K-2c_{0}K(\frac{S_{r}}{K})^{\frac{2^{*r}}{2}}D(u+\varphi)(x_{0})\varepsilon^{\frac{N-2r}{2}}+o(\varepsilon^{\frac{N-2r}{2}}),
\end{equation*}
Therefore
\begin{equation*}
S_{\theta,r}(\varphi)- \|u\|_{r}^{2}<S_{r}\left[1-\|u+\varphi\|_{2^{*r}}^{2^{*r}}\right]^{\frac{2}{2^{*r}}}.
\end{equation*}
\end{itemize}
\textbf{If $r$ is odd}\\
We use again $u+c_{\varepsilon} u_{x_{0},\varepsilon}$ as a testing function of \eqref{eq2}, we have
\begin{equation*}
S_{\theta,r}(\varphi)\leq \|u\|_{r}^{2}+c_{\varepsilon}^{2}\displaystyle{\int_{\Omega} \vert \nabla(-\Delta)^{\frac{r-1}{2}}u_{x_{0},\varepsilon}\vert^{2}dx}+2 c_{\varepsilon} \displaystyle{\int_{\Omega} \vert\nabla(-\Delta)^{\frac{r-1}{2}} u_{x_{0},\varepsilon}\vert \vert \nabla(-\Delta)^{\frac{r-1}{2}} u\vert dx.}
\end{equation*}
Using (\ref{eq*}) and applying the same technics used in the case where $r$ is even, we obtain
 \begin{equation*}
S_{\theta,r}(\varphi)- \|u\|_{r}^{2}<S_{r}\left[1-\|u+\varphi\|_{2^{*r}}^{2^{*r}}\right]^{\frac{2}{2^{*r}}}.
\end{equation*}
Note that if instead of $(u+\varphi)(x_{0})>0$ we had $(u+\varphi)(x_{0})<0$, then we would choose $c_{\varepsilon}>0$ such that $\|u+\varphi-c_{\varepsilon}u_{x_{0},\varepsilon}\|_{L^{2^{*}}}=1$.
Which completes the proof of Step 4 and then the proof of Theorem \ref{th1} is done. \hfill $\square$ \\

\begin{remark}
Let us note that any minimizers $u_\theta \in H_{\theta}^{r}(\Omega)$ of $S_{\theta,r}$, respectively $u_0 \in H_{0}^{r}(\Omega)$ of $S_{0,r}$,
satisfy the following Euler-Lagrange equations:
\begin{equation} \label{BN}
 \left\{
 \begin{array}{lclll}
(-\Delta)^{r}u_{\theta} = \Lambda_{\theta} \vert u_{\theta}+ \varphi \vert^{2^{*r}-2}(u_{\theta}+ \varphi)  &\mbox{in} & \Omega, \\
  \Delta^{r-1} u_{\theta}=.....=\Delta u_{\theta}  = u_{\theta}= 0 &\mbox{on} & \partial \Omega, \\
 \end{array}
 \right.
\end{equation}

and
\begin{equation} \label{E1}
\left\{
\begin{array}{lclll}
(-\Delta)^{r}u_{0} = \Lambda_{0} \vert u_{0}+ \varphi \vert^{2^{*r}-2}(u_{0}+ \varphi)  &\mbox{in} & \Omega, \\
 \frac{\partial^{r-1} u_{0} }{{(\partial \nu)}^{r-1}} =....= \frac{\partial u_{0} }{{\partial \nu}}= u_{0}= 0 &\mbox{on} & \partial \Omega, \\
\end{array}
  \right.
\end{equation}
where $\Lambda_{\theta} $ is the Lagrange multiplier associated to $ u_{\theta}$ and $\Lambda_{0} $ is the Lagrange multiplier associated to $ u_{0}$.
\label{rm1}
\end{remark}
 By analogy of the case $r=2$ in \cite{HL}, we can find the sign of the Lagrange multipliers which depends on $ \|\varphi\|_{L^{2^{*r}}}$ and we have
 \begin{proposition}~\\
 \begin{itemize}
 \item[(a)] If $\|\varphi\|_{L^{2^{*r}}} < 1 $  then $\Lambda_{\theta} > 0. $ and  $\Lambda_{0} > 0. $\\
 \item[(b)] If $\|\varphi\|_{L^{2^{*r}}} > 1 $  then $\Lambda_{\theta} < 0. $ and $\Lambda_{0} < 0. $\\
 \end{itemize}
 \label{rm2}
 \end{proposition}
 \begin{proof}
 We shall prove the results for $\Lambda_{\theta}$, the proof of results for $\Lambda_{0}$ are similar.\\
 We begin by noticing that $ \Lambda_{\theta}$ can be written as:
 \begin{equation} \label{Re1}
  S_{\theta,r}= \Lambda_{\theta} \left[ 1- \int_{\Omega} \vert u_{\theta}+ \varphi \vert^{2^{*r}-2}(u_{\theta}+ \varphi)  \varphi \right].
  \end{equation}
 Indeed,
 we have (see Remark~{\ref{rm1}})
 $$ S_{\theta,r}= \Lambda_{\theta} \int_{\Omega} \vert u_{\theta}+ \varphi \vert^{2^{*r}-2}(u_{\theta}+ \varphi) u_\theta, $$
 and \\
 $\displaystyle{ \int_{\Omega} \vert u_{\theta}+ \varphi \vert^{2^{*r}-2}(u_{\theta}+ \varphi)(u_{\theta}+ \varphi) = \int_{\Omega} \vert u_{\theta}+ \varphi \vert^{2^{*r}-2}(u_{\theta}+ \varphi) u_\theta + \int_{\Omega} \vert u_{\theta}+ \varphi \vert^{2^{*r}-2}(u_{\theta}+ \varphi)  \varphi,} $ \\
 and since,
 $$ \int_{\Omega} \vert u_{\theta}+ \varphi \vert^{2^{*r}} = 1. $$
Therefore we deduce \eqref{Re1}. \\
 Then, if we suppose $\|\varphi\|_{L^{2^{*r}}} < 1, $ and by the H\"older inequality we have
  $$ \int_{\Omega} \vert u_{\theta}+ \varphi \vert^{2^{*r}-2}(u_{\theta}+ \varphi)  \varphi \leq \left[  \int_{\Omega} \left( \vert u_{\theta}+ \varphi \vert^{2^{*r}-1}\right)^{\frac{2^{*r}}{2^{*r}-1}} dx \right]^{\frac{2^{*r}-1}{2^{*r}}} \left[ \int_{\Omega} {\vert \varphi \vert}^{2^{*r}} dx \right]^{\frac{1}{2^{*r}}}. $$
 Since $u \not\equiv 0$ except for $ \|\varphi\|_{L^{2^{*r}}} = 1 $ which is an obvious case. \\
 Thus,
 $$ \int_{\Omega} \vert u_{\theta}+ \varphi \vert^{2^{*r}-2}(u_{\theta}+ \varphi)  \varphi \leq \left[ \int_{\Omega} {\vert \varphi \vert}^{2^{*r}} dx \right]^{\frac{1}{2^{*r}}} < 1 $$
 and then $ \Lambda_{\theta} > 0. $ Now we assume that $\|\varphi\|_{L^{2^{*r}}} > 1 $ and set, as in \cite{HL},
$$ h(t)= \int_{\Omega} \vert t u_{\theta}+ \varphi \vert^{2^{*r}}. $$
This function admits a derivative given by the formula
$$ h'(t)= 2^{*r} \int_{\Omega} \vert t u_{\theta}+ \varphi \vert^{2^{*r}-2} (tu_{\theta}+ \varphi) u_\theta. $$
Now, the function $u_\theta$ satisfies
$$ (-\Delta)^{r}u_{\theta} = \Lambda_{\theta} \vert u_{\theta}+ \varphi \vert^{2^{*r}-2}(u_{\theta}+ \varphi). $$
Then, multiplying  by $ u_{\theta} $ and integrating by parts, we get \\
\textbf{If $r$ is even}
$$ \int_{\Omega} {\vert (-\Delta)^{\frac{r}{2}} u_{\theta} \vert}^{2} dx= \Lambda_{\theta} \int_{\Omega} \vert u_{\theta}+ \varphi \vert^{2^{*r}-2} (u_{\theta}+ \varphi) u_\theta= \frac{\Lambda_{\theta}}{2^{*r}}h'(1). $$
\textbf{If $r$ is odd}
$$\displaystyle{ \int_{\Omega} |\nabla(-\Delta)^{\frac{r-1}{2}} u_{\theta} |^{2}dx} = \Lambda_{\theta} \int_{\Omega} \vert u_{\theta}+ \varphi \vert^{2^{*r}-2} (u_{\theta}+ \varphi) u_\theta= \frac{\Lambda_{\theta}}{2^{*r}}h'(1).$$
So item $(b)$ is verified because $h(1)=1, $ and we see that $h(t)\geq 1$ for all $t \in \left[ 0, 1 \right]. $ So, we conclude that $ h'(1)\leq 0. $ Otherwise, since $h$ is continuous and $ h(0) > 1, $ there exists $ 0<s<1$ such that $ h(s)=1. $\\
Therefore,
$$\int_{\Omega} \vert s u_{\theta}+ \varphi \vert^{2^{*r}} = 1. $$
When $s u_\theta $ as a testing function in(2), we have  \\
\textbf{If $r$ is even}
$$S_{\theta,r}= \displaystyle{ \int_{\Omega} |(-\Delta)^{r/2} s u_\theta |^{2}dx} \leq s^{r} \displaystyle{ \int_{\Omega} |(-\Delta)^{r/2} u_{\theta}  |^{2}dx}. $$
\textbf{If $r$ is odd}
$$S_{\theta,r}= \displaystyle{ \int_{\Omega} |\nabla(-\Delta)^{\frac{r-1}{2}} su_{\theta} |^{2}dx} \leq s^{r-1} \displaystyle{ \int_{\Omega} |\nabla(-\Delta)^{\frac{r-1}{2}} u_{\theta} |^{2}dx}. $$

We get a contradiction and the proof is completed.
 \end{proof}
 \begin{remark}
 In \cite{G}, the author considered the following semi-linear polyharmonic problem:
 \begin{equation} \label{G}
 \left\{
 \begin{array}{lclll}
(-\Delta)^{r}u =  \vert u \vert^{2^{*r}-2} u + f(x,u) &\mbox{in} & \Omega, \\
u> 0 & \mbox{in} & \Omega, \\
  (-\Delta)^{r-1} u =.....=(-\Delta) u  = u = 0 &\mbox{on} & \partial \Omega. \\
 \end{array}
 \right.
\end{equation}
This problem is equivalent to \eqref{BN} when $\Lambda_{\theta} > 0 $ is fixed. The author prove the existence of positive solutions under the sufficient conditions on $f$ and the domain $\Omega$.
\label{rm3}
 \end{remark}
\section{Proof of Theorem 2}
By definitions of \eqref{eq1} and \eqref{eq2} we have $ S_{\theta,r} \leq S_{0,r}. $ In this section we present a gap phenomenon between $ S_{\theta,r}$ and $S_{0,r}$ under suitable hypothesis on $ \varphi. $ \\
Proof of $(i). $\\
Let  $\varphi$ be a positive function not identically zero. We adapt the argument of Van der Vorst \cite{V2} to the present situation. Let $ u_{\theta}$ a the minimizer of \eqref{eq2}. We give reason by contradiction. We assume that $ u_{\theta}$ is in $ H_0^r(\Omega). $ \\
\textbf{If $r$ is even}\\
Let $v$ be the solution of the following problem
\begin{equation} \label{E3}
 \left\{
       \begin{array}{lclll}
 (-\Delta)^{\frac{r}{2}}v =  \vert  (-\Delta)^{\frac{r}{2}} u_{\theta} \vert  &\mbox{in} & \Omega, \\
  (-\Delta)^{\frac{r}{2}-1} v=.....= -\Delta v = v= 0 &\mbox{on} & \partial \Omega, \\
 \end{array}
    \right.
\end{equation}
We get
\begin{equation}
 \left\{
       \begin{array}{lclll}
 (-\Delta)(-\Delta)^{\frac{r}{2}-1}(v - u_{\theta}) \geq 0  &\mbox{in} & \Omega, \\
 (-\Delta)^{\frac{r}{2}-1}(v - u_{\theta})  = 0 &\mbox{on} & \partial \Omega,
 \end{array}
    \right.
    \label{E4}
\end{equation}
 and
\begin{equation}
 \left\{
       \begin{array}{lclll}
 (-\Delta)(-\Delta)^{\frac{r}{2}-1}(v + u_{\theta}) \geq 0  &\mbox{in} & \Omega, \\
(-\Delta)^{\frac{r}{2}-1}(v + u_{\theta}) = 0 &\mbox{on} & \partial \Omega.
 \end{array}
    \right.
    \label{E5}
\end{equation}
In equations \eqref{E4} and \eqref{E5}, using successively the maximum principle we obtain  $ v > \vert u_{\theta} \vert $ or $ v= - u_{\theta}$ or $ v= u_{\theta}. $ \\
By taking the equation \eqref{E3} with $ v= u_{\theta}$ and $ v= - u_{\theta}, $ we find the function $ (-\Delta)^{\frac{r}{2}} u_{\theta}$ has a constant sign. These two cases $ v= u_{\theta} $ or $- u_{\theta}$ when $ u_{\theta}= \frac{\partial  u_{\theta} }{\partial \nu}=...=\frac{\partial^{\frac{r}{2} -1} u_{\theta} }{(\partial \nu)^{\frac{r}{2} -1}}  = 0 \ \mbox{on} \ \partial \Omega $ are false if we use the maximum principle when we consider $ u_{\theta} = 0$ in $\Omega$. So we have $ v > \vert u_{\theta} \vert $ in $ \Omega. $ \\ Considering this inequality and the fact that $ \varphi \geq 0, $ we get $  u_{\theta} + \varphi < v + \varphi $ in $ \Omega$ and $ -u_{\theta} - \varphi < v + \varphi $ in $ \Omega; $ therefore $ \vert u_{\theta} + \varphi \vert < \vert v + \varphi \vert $ in $ \Omega $ and as result we have
$$\displaystyle \int_{\Omega}|v +\varphi|^{2^{*r}} dx > 1. $$
Currently, take the function $ f(t)= \displaystyle{ \int_{\Omega}|tv +\varphi|^{2^{*r}} dx} $ for $ t \in \left[ 0, 1 \right]. $ Since $f$ is continuous, $ f(0)< 1 $ and $ f(1)> 1, $ there exists $ s \in \left] 0, 1 \right[ $ such that $ f(s) =1. $\\
But we have
$$ \int_{\Omega} {\vert (- \Delta)^{\frac{r}{2}}s v \vert}^{2} dx \leq s^2 \int_{\Omega} {\vert (- \Delta)^{\frac{r}{2}} v \vert}^{2} dx,  $$
this gives a contradiction with the definition of $ S_{\theta,r}(\varphi). $ \\
\textbf{If $r$ is odd} \\
Let $v$ be the solution of the next problem:
\begin{equation} \label{E6}
 \left\{
       \begin{array}{lclll}
 (-\Delta)^{\frac{r-1}{2}}v =  \vert  (-\Delta)^{\frac{r-1}{2}} u_{\theta} \vert  &\mbox{in} & \Omega, \\
  (-\Delta)^{\frac{r-1}{2}-1} v=.....= -\Delta v = v= 0 &\mbox{on} & \partial \Omega, \\
 \end{array}
    \right.
\end{equation}
We obtain
\begin{equation} \label{E7}
 \left\{
       \begin{array}{lclll}
 (-\Delta)(-\Delta)^{\frac{r-1}{2}-1}(v - u_{\theta}) \geq 0  &\mbox{in} & \Omega, \\
 (-\Delta)^{\frac{r-1}{2}-1}(v - u_{\theta})  = 0 &\mbox{on} & \partial \Omega, \\
 \end{array}
    \right.
\end{equation}
 and
\begin{equation} \label{E8}
 \left\{
       \begin{array}{lclll}
 (-\Delta)(-\Delta)^{\frac{r-1}{2}-1}(v + u_{\theta}) \geq 0  &\mbox{in} & \Omega, \\
(-\Delta)^{\frac{r-1}{2}-1}(v + u_{\theta}) = 0 &\mbox{on} & \partial \Omega. \\
 \end{array}
    \right.
\end{equation}
In \eqref{E7} and \eqref{E8}, using successively the maximum principle we obtain  $ v > \vert u_{\theta} \vert $ or $ v= - u_{\theta}$ or $ v= u_{\theta}. $ \\
By taking the equation \eqref{E6} with $ v= u_{\theta}$ and $ v= - u_{\theta}, $ we find the function $ (-\Delta)^{\frac{r-1}{2}} u_{\theta}$ has a constant sign. These two cases $ v= u_{\theta} $ or $- u_{\theta}$ when $ u_{\theta}= \frac{\partial  u_{\theta} }{\partial \nu}=...=\frac{\partial^{\frac{r-1}{2} -1} u_{\theta} }{(\partial \nu)^{\frac{r-1}{2} -1}}  = 0 \ \mbox{on} \ \partial \Omega $ are false if we use the maximum principle when we consider $ u_{\theta} = 0$ in $\Omega$. Thus, we have $ v > \vert u_{\theta} \vert $ in $ \Omega. $ \\ Using this inequality and the fact that $ \varphi \geq 0, $ we get $  u_{\theta} + \varphi < v + \varphi $ in $ \Omega$ and $ -u_{\theta} - \varphi < v + \varphi $ in $ \Omega; $ therefore $ \vert u_{\theta} + \varphi \vert < \vert v + \varphi \vert $ in $ \Omega $ and as result we have
$$\displaystyle \int_{\Omega}|v +\varphi|^{2^{*r}} dx > 1. $$
Currently, let us consider the function $ f(t)= \displaystyle{ \int_{\Omega}|tv +\varphi|^{2^{*r}} dx} $ for $ t \in \left[ 0, 1 \right]. $ Since $f$ is continuous, $ f(0)< 1 $ and $ f(1)> 1, $ there exists $ s \in \left] 0, 1 \right[ $ such that $ f(s) =1. $\\
But we have
$$ \int_{\Omega} {\vert \nabla( (- \Delta)^{\frac{r-1}{2}}s v) \vert}^{2} dx \leq s^2 \int_{\Omega} {\vert (- \Delta)^{\frac{r-1}{2}} v \vert}^{2} dx,  $$
that contradiction the definition of $ S_{\theta,r}(\varphi). $  \\
 This finish the proof of $(i)$. \\

\textbf{Proof of (ii).}
We will prove it into two cases. \\
\textbf{Case 1:} Assume that $ \varphi$ is in $(H^{r}_{0}(\Omega))^{\perp}$ and $\|\varphi\|_{L^{2^{*r}}} > 1$.  Let $ u_{\theta} $ a solution of \eqref{eq2}. Multiplying \eqref{BN} by $ u_{\theta}+ \varphi $ and integrating by parts, we get \\
\textbf{If $r$ is even}
$$ \int_{\Omega} {\vert (-\Delta)^{\frac{r}{2}} u_{\theta} \vert}^{2} + \int_{\Omega} (-\Delta)^{\frac{r}{2}}  u_{\theta} \cdot  (-\Delta)^{\frac{r}{2}} \varphi dx = \Lambda_{\theta}. $$

\textbf{If $r$ is odd} \\
We have
$$ \int_{\Omega} |\nabla(-\Delta)^{\frac{r-1}{2}} u_{\theta} |^{2}dx  + \int_{\Omega} \nabla(-\Delta)^{\frac{r-1}{2}} u_{\theta} \cdot  \nabla(-\Delta)^{\frac{r-1}{2}} \varphi dx = \Lambda_{\theta}. $$
As $\|\varphi\|_{L^{2^{*r}}} > 1,$ we have  $ \Lambda_{\theta} < 0, $ therefore $ \displaystyle{ \int_{\Omega} (-\Delta)^{\frac{r}{2}}  u_{\theta} \cdot  (-\Delta)^{\frac{r}{2}} \varphi dx < 0, } $ if $r$ is even and $\displaystyle{ \int_{\Omega} \nabla(-\Delta)^{\frac{r-1}{2}} u_{\theta} \cdot  \nabla(-\Delta)^{\frac{r-1}{2}} \varphi dx<0}$ if $r$ is odd. Which improve that $u_\theta $ is not in $ H^{r}_{0}(\Omega); $ in conclusion we have $ \displaystyle{S_{\theta,r}(\varphi) < S_{0,r}(\varphi). }$ \\
\textbf{Case 2:} Assume that $ \varphi$ is in $(H^{r}_{0}(\Omega))^{\perp}$ and $\|\varphi\|_{L^{2^{*r}}} < 1$. Let $x_{0} \in \Omega$ and $u_{x_{0},\varepsilon}$ defined in (\ref{D}). From \cite{EFJ}, we have $ \displaystyle{\int_{\Omega} {\vert u_{x_{0}, \varepsilon} \vert}^{2^{*r}} = \frac{K}{S_{r}}+ o(1), }$ and $ \displaystyle{\int_{\Omega} {\vert \Delta u_{x_{0}, \varepsilon} \vert}^{2} = K + o(1), }$ where  $K$ is a positive constant. \\
Since $\|\varphi\|_{L^{2^{*r}}} < 1,$ there exists $ c_\varepsilon > 0 $ such that $ \|\varphi + c_\varepsilon u_{x_{0}, \varepsilon} \|_{L^{2^{*r}}} = 1. $ By Brezis-Lieb identity (see \cite{BL} ) we have
$$ c_{\varepsilon}^{2^{*r}} = (\frac{S_{r}}{K})^{\frac{2^{*r}}{2}} \left[ 1 - \| \varphi \|_{L^{2^{*r}}}^{2^{*r}} \right]+o(1), $$
then
$$ c_{\varepsilon}^{2} K = S_{r} \left[ 1 -\| \varphi \|_{L^{2^{*r}}}^{2^{*r}} \right]^{\frac{N-2r}{N}}+o(1).   $$
when $o(1)$ tends to $ 0. $
At limit, we have
\begin{equation}
 c_{\varepsilon}^{2} K = S_{r} \left[ 1 -\| \varphi \|_{L^{2^{*r}}}^{2^{*r}} \right]^{\frac{N-2r}{N}}.
\label{eqdimanche1}
\end{equation}
 Afterwards, since $\|\varphi+c_{\varepsilon} u_{x_{0},\varepsilon} \|_{L^{2^{*r}}}^{2}=1$ we write
$$ S_{\theta,r}(\varphi) \leq c_{\varepsilon}^{2} \| u_{a, \varepsilon}\|_{r}^{2} = S_{r} \left[ 1 - \| \varphi \|_{L^{2^{*r}}}^{2^{*r}} \right]^{\frac{N-2r}{N}}+o(1). $$
Using (\ref{eqdimanche1}), direct computations show that
\begin{equation} \label{Eps}
S_{\theta,r}(\varphi) \leq S_{r} \left(  1 - \| \varphi \|_{L^{2^{*r}}}^{2^{*r}} \right)^{\frac{N-2r}{N}}.
\end{equation}
On the other way, multiplying \eqref{BN} by $u_{\theta}$ and integrating, we get
$$ S_{\theta,r}(\varphi) = \Lambda_\theta \int_{\Omega} \vert u_{\theta}+ \varphi \vert^{2^{*r}-2}(u_{\theta}+ \varphi) u_{\theta}dx. $$
Thus, the H\"older inequality gives
$$ S_{\theta,r}(\varphi) \leq \Lambda_\theta \left[  \int_{\Omega} \left( \vert u_{\theta}+ \varphi \vert^{2^{*r}-1}\right)^{\frac{2^{*r}}{2^{*r}-1}} dx \right]^{\frac{2^{*r}-1}{2^{*r}}} \left[ \int_{\Omega} {\vert u_{\theta} \vert}^{2^{*r}} dx \right]^{\frac{1}{2^{*r}}} $$
\begin{equation} \label{eq}
 S_{\theta,r}(\varphi) \leq \Lambda_\theta \| u_\theta \|_{L^{2^{*r}}}.
\end{equation}
Applying the Sobolev inequality we obtain that
\begin{equation} \label{S}
 S_{\theta,r}(\varphi) \leq \Lambda_\theta \frac{1}{S_{r}^{\frac{1}{2}}} \|u\|_{r}.
\end{equation}
Or
\begin{equation}
\| u_{\theta}\|_{r}^{2}=S_{\theta,r}(\varphi).
\label{eqdimanche2}
\end{equation}
Combining (\ref{eqdimanche2}) and \eqref{S} we find
\begin{equation} \label{Def}
S_{\theta,r}(\varphi) \leq \Lambda_{\theta} \left[ 1 - \| \varphi \|_{L^{2^{*r}}}^{2^{*r}} \right]^{\frac{1}{2^{*r}}}.
\end{equation}
Now, multiplying \eqref{BN} by $ (u_\theta + \varphi) $ and integrating, using (\ref{Def}) we acquire\\
\textbf{If $r$ is even}
\begin{equation} \label{Int}
\int_{\Omega} (-\Delta)^{\frac{r}{2}}  u_{\theta} \cdot  (-\Delta)^{\frac{r}{2}} \varphi dx = \Lambda_\theta - S_{\theta,r}(\varphi).
\end{equation}
Combining \eqref{Def} and \eqref{Int} we are lead to
$$ \int_{\Omega} (-\Delta)^{\frac{r}{2}}  u_{\theta} \cdot  (-\Delta)^{\frac{r}{2}} \varphi dx \geq \Lambda_\theta \left[ 1- \left( 1- \int_{\Omega} {\vert \varphi \vert}^{2^{*r}} \right)^{\frac{N-2r}{2N}} \right] > 0, $$
\textbf{If $r$ is odd} \\
\begin{equation} \label{od}
\int_{\Omega}  \nabla (-\Delta)^{\frac{r-1}{2}}  u_{\theta} \cdot  \nabla (-\Delta)^{\frac{r-1}{2}}  \varphi dx = \Lambda_\theta - S_{\theta,r}(\varphi).
\end{equation}
Combining \eqref{Def} and \eqref{od} we have
$$ \int_{\Omega}  \nabla (-\Delta)^{\frac{r-1}{2}}  u_{\theta} \cdot  \nabla (-\Delta)^{\frac{r-1}{2}}  \varphi dx \geq \Lambda_\theta \left[ 1- \left( 1- \int_{\Omega} {\vert \varphi \vert}^{2^{*r}} \right)^{\frac{N-2r}{2N}} \right] > 0. $$
This means that $ u_\theta $ is not in $ H^{r}_{0}(\Omega) $ and we have $ S_{\theta,r}(\varphi)< S_{0,r}(\varphi)$. \\
Indeed, let us note that any minimizer $ u_\theta \in H^{r}_{\theta}(\Omega)$ of $S_{\theta,r}(\varphi)$ is not in $ H^{r}_{0}(\Omega). $ \\
Arguing by contradiction, suppose that $ S_{\theta,r}(\varphi) = S_{0,r}(\varphi), $ thus $ S_{0,r}(\varphi)= S_{\theta,r}(\varphi) = \|u \|_{r}^{2} $ and $ \|u + \varphi \|_{L^{2^{*r}}} = 1. $ Therefore $ u_\theta $ be a minimizer of $ S_{0,r}(\varphi), $ as a result $ u_\theta \in H^{r}_{0}(\Omega), $ which gives a contradiction. \\
\textbf{Proof of (iii).}
Suppose $\varphi$ be in $ H^{r}_{0}(\Omega). $ We admit first that for $\|\varphi\|_{L^{2^{*r}}} > 1, $ \\
\textbf{If $r$ is even} \\
 we have,
\begin{equation} \label{Ph}
S_{\theta,r}(\varphi)= \inf_{ \substack{ u\in H_{\theta}^{r}(\Omega) \\ \|u+\varphi\|_{L^{2^{*r}}} \leq 1}} \displaystyle{ \int_{\Omega} |(-\Delta)^{\frac{r}{2}} u  |^{2}}dx .
\end{equation}
We find a convex problem. In this case, based on \cite {ET}, we will use a duality's method.\\
For all $ p \in L^2(\Omega),
$ \\
 Define
$$  \beta_{\theta} = \sup_{ \substack{ u\in H_{\theta}^{r}(\Omega) \\ \|u+\varphi\|_{L^{2^{*r}}} \leq 1}} \displaystyle{\int_{\Omega} p \left( (-\Delta)^{\frac{r}{2}} u \right) } \  \mbox{and} \ \beta_{0} = \sup_{\substack{ u\in H_{0}^{r}(\Omega) \\ \|u+\varphi\|_{L^{2^{*r}}} \leq 1}} \int_{\Omega} p \left( (-\Delta)^{\frac{r}{2}} u \right) .
$$
We get
$$  \beta_{\theta} = \sup_{\substack{ v\in H_{\theta}^{r}(\Omega) \\ \|v \|_{L^{2^{*r}}} \leq 1}} \int_{\Omega} p \left( (-\Delta)^{\frac{r}{2}} v\right)  - \int_{\Omega} p \left( (-\Delta)^{\frac{r}{2}} \varphi \right) \ \mbox{and} \ \beta_{0} = \sup_{\substack{ v\in H_{0}^{r}(\Omega) \\ \|v \|_{L^{2^{*r}}} \leq 1}} \int_{\Omega} p \left( (-\Delta)^{\frac{r}{2}} v\right)  - \int_{\Omega} p \left( (-\Delta)^{\frac{r}{2}} \varphi \right). $$
We will prove that we have
\begin{eqnarray} \label{B}
\hspace{-7cm} \beta_{\theta} = \beta_{0}, \ \ \ \mbox{for all}\ \   p \in L^2(\Omega)
\end{eqnarray}
Initially, we observe that $ \beta_{\theta}$ and  $ \beta_{0}$ are finite, due to the Holder inequality
 $$ \displaystyle\vert \int_{\Omega} p \left( (-\Delta)^{\frac{r}{2}} v \right)   \vert \leq \|p \|_{L^{2}} \|(-\Delta)^{\frac{r}{2}} v  \|_{L^{2}}.$$
  We include that the linear operator
$$ L: H_{\theta}^{r}(\Omega) \rightarrow \mathbb{R} $$
$$ v  \rightarrow  \int_{\Omega} p \left( (-\Delta)^{\frac{r}{2}} v \right)  $$
is continuous for the $ L^{2^{*r}} $ topology. So, there exists $ \tilde{p} \in L^{\frac{2 N}{N+2r}} $ even that for all $ v \in H_{\theta}^{r}(\Omega)$ we find $ L(v) = \int_{\Omega} \tilde{p} v. $\\
We consider that
 \begin{equation}
 \beta_{\theta} = \sup_{\substack{v \in H_{\theta}^{r}(\Omega) \\ \|v \|_{L^{2^{*r}}} \leq 1}} \int_{\Omega} \tilde{p} v - \int_{\Omega} p \left( -\Delta)^{\frac{r}{2}} \varphi \right) ; \ \ \ \beta_{0} = \sup_{\substack{ v \in H_{0}^{r}(\Omega) \\ \|v \|_{L^{2^{*r}}} \leq 1}} \int_{\Omega} \tilde{p} v - \int_{\Omega} p \left( (- \Delta)^{\frac{r}{2}} \varphi \right).
 \end{equation}
 On the other side, for all $ \tilde{p} \in L^{\frac{2 N}{N+2r}}, $ we have
 \begin{equation} \label{Eg}
 \sup_{ \substack{ v \in H_{\theta}^{r}(\Omega) \\ \|v \|_{L^{2^{*r}}} \leq 1}} \int_{\Omega} \tilde{p} v =\sup_{\substack{ v \in L^{2^{*r}}(\Omega) \\ \|v \|_{L^{2^{*r}}} \leq 1}} \int_{\Omega} \tilde{p} v   = \sup_{ \substack{ v \in H_{0}^{r}(\Omega) \\ \|v \|_{L^{2^{*r}}} \leq 1}} \int_{\Omega} \tilde{p} v = \|\tilde{p} \|_{L^{\frac{2 N}{N+2r}}}.
 \end{equation}
 Actually, in the case where $ \|\varphi\|_{L^{2^{*r}}} > 1, $ we have to prove that
 \begin{equation}  \label{MM}
 \frac{1}{2} S_{\theta,r}(\varphi) = \sup_{p \in L^{2}(\Omega) } \left\lbrace - \frac{1}{2} \int_{\Omega} {\vert p \vert}^{2} - \sup_{ \substack{u\in H_{\theta}^{r}(\Omega) \\ \|u+\varphi\|_{L^{2^{*r}}} \leq 1}} \int_{\Omega} p \left( (-\Delta)^{\frac{r}{2}} u\right) \right\rbrace
\end{equation}
and
\begin{equation} \label{MA}
 \frac{1}{2} S_{0,r}(\varphi) = \sup_{p \in L^{2}(\Omega) } \left\lbrace - \frac{1}{2} \int_{\Omega} {\vert p \vert}^{2} - \sup_{\substack{ u\in H_{0}^{r}(\Omega) \\ \|u+\varphi\|_{L^{2^{*r}}} \leq 1}} \int_{\Omega} p \left( (-\Delta)^{\frac{r}{2}} u \right) \right\rbrace.
\end{equation}
Let us write, for $ p \in L^2(\Omega) $ and for $ u \in H_{\theta}^{r}(\Omega), $
$$ L(u, p)= - \frac{1}{2} \int_{\Omega} {\vert p \vert}^{2}- \int_{\Omega} \left( (-\Delta)^{\frac{r}{2}} u \right)  p. $$
We show that
\begin{eqnarray} \label{L}
 \sup_{p \in L^{2}(\Omega)} L(u,p) = \frac{1}{2} \int_{\Omega} {\vert (-\Delta)^{\frac{r}{2}} u \vert}^{2}.
\end{eqnarray}
Indeed, \\
Firstly, we have
$$ L(u, (-\Delta)^{\frac{r}{2}} u)= - \frac{1}{2} \int_{\Omega} {\vert (-\Delta)^{\frac{r}{2}} u  \vert}^{2} + \int_{\Omega} \left( (-\Delta)^{\frac{r}{2}} u \right)^2 , $$
thus,
$$ L(u, (-\Delta)^{\frac{r}{2}} u) = \frac{1}{2} \int_{\Omega} {\vert (-\Delta)^{\frac{r}{2}} u \vert}^{2}. $$
Therefore,
\begin{equation}
\sup_{p \in L^{2}(\Omega)} L(u,p) \geq L(u, (-\Delta)^{\frac{r}{2}} u).
\label{eqdimanche3}
\end{equation}
On the other hand,
we have,
$$ \frac{1}{2} \|(-\Delta)^{\frac{r}{2}} u+ p\|_{L^{2}}^{2} = \frac{1}{2} \| (-\Delta)^{\frac{r}{2}} u \|_{L^{2}}^{2} + \frac{1}{2} \|p \|_{L^{2}}^{2} + \int_{\Omega} ((-\Delta)^{\frac{r}{2}} u) p , $$
then,
$$ L(u, p)= - \frac{1}{2} \int_{\Omega} {\vert p \vert}^{2}- \int_{\Omega} ((-\Delta)^{\frac{r}{2}} u) p \leq - \frac{1}{2} \int_{\Omega} {\vert p \vert}^{2}+ \frac{1}{2} \int_{\Omega} {\vert (-\Delta)^{\frac{r}{2}} u + p \vert}^{2} dx - \int_{\Omega} ((-\Delta)^{\frac{r}{2}} u) p. $$
Thus
\begin{equation}
\sup_{p \in L^{2}(\Omega)} L(u,p) \leq \frac{1}{2} \int_{\Omega} {\vert (-\Delta)^{\frac{r}{2}} u  \vert}^{2}.
\label{eqdimanche4}
\end{equation}
Combining (\ref{eqdimanche3}) and (\ref{eqdimanche4}), we get (\ref{L}).\\
Now, by (\ref{Ph}) and (\ref{L}) we have
 \begin{eqnarray*}
  \ \ \ \ \  \frac{1}{2} S_{\theta,r}(\varphi) =  \inf_{ \substack{ u\in H_{\theta}^{r}(\Omega) \\ \|u+\varphi\|_{L^{2^{*r}}} \leq 1}}\sup_{p \in L^{2}(\Omega) } L(u,p). \hspace{8cm} (Y)
 \end{eqnarray*}
 Let us justify that $(Y)= (Y^{*})$ where $(Y^{*})$ is the dual problem of $ (Y), $ defined by
 \begin{eqnarray*}
 \ \ \  \sup_{p \in L^{2}(\Omega) } \inf_{ \substack{ u\in H_{\theta}^{r}(\Omega) \\ \|u+\varphi\|_{L^{2^{*r}}} \leq 1}}  L(u,p) . \hspace{10cm} \ \  (Y^{*})
 \end{eqnarray*}
 Let us define \\
 $$  \hspace{-8cm} A= \left\lbrace u\in H_{\theta}^{r}(\Omega); \|u+\varphi\|_{L^{2^{*r}}} \leq 1 \right\rbrace. $$
Let $ u_{\theta}$ be a minimizer that realizes $S_{\theta,r}(\varphi) $ and let  $ p_{\theta} = -(- \Delta)^{\frac{r}{2}} u_{\theta}. $ From \eqref{L}, we get
 \begin{eqnarray} \label{eqdimanche5}
 L(u_{\theta},p) \leq \sup_{p \in L^{2}(\Omega)} L(u_{\theta},p) = \frac{1}{2} S_{\theta,r}(\varphi)  \ \ \ \mbox{for all} \ \ \ p \in L^{2}(\Omega).
\end{eqnarray}
Actually, let $ u \in A $ we have
 $$ L(u,p_{\theta}) \geq  - \frac{1}{2} \int_{\Omega} {\vert p_{\theta} \vert}^{2}- \sup_{u \in A} \int_{\Omega} (-\Delta u )^{\frac{r}{2}} p_{\theta}. $$
 Indeed, by definition we write
 \begin{equation} \label{I}
 \sup_{u \in A} \int_{\Omega} ((-\Delta)^{\frac{r}{2}} u ) p_{\theta} \geq  \int_{\Omega} ((-\Delta)^{\frac{r}{2}}u) p_{\theta}.
 \end{equation}
 Multiplying \eqref{I} by $-1, $ we get
 \begin{equation} \label{2A}
 - \sup_{u \in A} \int_{\Omega} ((-\Delta)^{\frac{r}{2}} u ) p_{\theta} \leq  - \int_{\Omega} ((-\Delta)^{\frac{r}{2}} u ) p_{\theta},
 \end{equation}
we add $\displaystyle{ - \frac{1}{2} \int_{\Omega} {\vert p_{\theta} \vert}^{2}dx} $ in \eqref{2A}, we obtain
 $$ - \frac{1}{2} \int_{\Omega} {\vert p_{\theta} \vert}^{2} - \sup_{u \in A} \int_{\Omega} ((-\Delta)^{\frac{r}{2}} u ) p_{\theta}  \leq - \frac{1}{2} \int_{\Omega} {\vert p_{\theta} \vert}^{2} - \int_{\Omega} ((-\Delta)^{\frac{r}{2}} u ) p_{\theta}. $$
 Therefore,
 $$ - \frac{1}{2} \int_{\Omega} {\vert p_{\theta} \vert}^{2} dx - \sup_{u \in A} \int_{\Omega} ((-\Delta)^{\frac{r}{2}} u ) p_{\theta} dx \leq L(u,p_{\theta}). $$
Now, integrating by part and using the fact that $ (-\Delta)^{\frac{r}{2}} u_{\theta} = 0$ on $ \partial \Omega, $, we have
$$ L(u,p_{\theta}) \geq  - \frac{1}{2} \int_{\Omega} {\vert p_{\theta} \vert}^{2}- \sup_{u \in A} \int_{\Omega}  u ((-\Delta)^{\frac{r}{2}} p_{\theta}). $$
On the other hand, using (\ref{Eg}), we see that
$$
 \begin{array}{cll}
\displaystyle{ - \sup_{u \in A} \int_{\Omega}  u ((-\Delta)^{\frac{r}{2}} p_{\theta})dx } &=& \displaystyle{- \sup_{u \in A} \int_{\Omega}  (u + \varphi) (-\Delta)^{\frac{r}{2}} p_{\theta})dx + \int_{\Omega} \varphi ((-\Delta)^{\frac{r}{2}} p_{\theta})dx }\\
 &=& \displaystyle{ - \| (-\Delta)^{\frac{r}{2}} p_{\theta} \|_{L^{\frac{2 N}{N+2r}}} + \int_{\Omega} \varphi ((-\Delta)^{\frac{r}{2}} p_{\theta})dx }\\
& =& \displaystyle{ - \| (-\Delta)^{\frac{r}{2}} p_{\theta} \|_{L^{\frac{2 N}{N+2r}}} + \int_{\Omega} ((-\Delta)^{\frac{r}{2}} \varphi) p_{\theta} - \int_{ \partial \Omega} \varphi ((-\Delta)^{\frac{r}{2}} p_{\theta})dx }. \\
\end{array}
$$
Using $ (-\Delta)^{\frac{r}{2}} p_{\theta} = 0 $ on $ \partial \Omega $ we obtain,
$$\displaystyle{ - \sup_{u \in A} \int_{\Omega}  u ((-\Delta)^{\frac{r}{2}} p_{\theta})dx }= \displaystyle{ - \| (-\Delta)^{\frac{r}{2}} p_{\theta} \|_{L^{\frac{2 N}{N+2r}}} + \int_{\Omega} ((-\Delta)^{\frac{r}{2}} \varphi) p_{\theta}},$$
and therefore
\begin{equation} \label{No}
L(u,p_{\theta}) \geq  - \frac{1}{2} \int_{\Omega} {\vert p_{\theta} \vert}^{2}\displaystyle{ - \| (-\Delta)^{\frac{r}{2}} p_{\theta} \|_{L^{\frac{2 N}{N+2r}}} + \int_{\Omega} ((-\Delta)^{\frac{r}{2}} \varphi) p_{\theta}.}
\end{equation}
However, the Euler equation \eqref{BN} for $ u_\theta$ gives
\begin{equation} \label{La}
\| (-\Delta)^r u_{\theta} \|_{L^{\frac{2 N}{N+2r}}} = \vert \Lambda_{\theta} \vert.
\end{equation}
Contrastingly, multiplying the Euler equation \eqref{BN} by $  (u_\theta+ \varphi), $  we obtain
\begin{equation} \label{Eeq}
\Lambda_{\theta} = \int_{\Omega} {\vert (-\Delta)^{\frac{r}{2}} u_{\theta} \vert}^{2} +  \int_{\Omega} (-\Delta)^{\frac{r}{2}} u_{\theta} \cdot (-\Delta^{\frac{r}{2}}) \varphi dx.
\end{equation}
From Remark~\ref{rm2} we know that $ \Lambda_{\theta}<0$ since $\|\varphi\|_{L^{2^{*r}}}<1$,  therefore \eqref{La} and \eqref{Eeq} give
  \begin{equation} \label{De}
 -\| (-\Delta)^{\frac{r}{2}} p_{\theta} \|_{L^{\frac{2 N}{N+2r}}} +  \int_{\Omega} p_{\theta} \cdot (-\Delta)^{\frac{r}{2}} \varphi dx = S_{\theta,r}(\varphi) .
\end{equation}
We replace \eqref{No} into \eqref{De}, we find
\begin{equation} \label{NO}
L(u,p_{\theta}) \geq  \frac{1}{2} S_{\theta,r}(\varphi) \ \ \ \mbox{for all } \ u \in A.
\end{equation}
Regarding to \cite{ET} that \eqref{NO} and \ref{eqdimanche5} conclude that $ (Y) = (Y^{*}). $
This proof is valid for $ \frac{1}{2} S_{\theta,r}(\varphi) $ instead of $ \frac{1}{2} S_{0,r}(\varphi), $ then we have proved \eqref{MM} and \eqref{MA}. Consequently, we have \eqref{B} and then we conclude that $ S_{\theta,r}(\varphi)=S_{0,r}(\varphi). $ \\
\textbf{If $r$ is odd} \\
In the case when $r $ is odd we find the same results just we have
 $$S_{\theta,r}(\varphi)= \inf_{ \substack{ u\in H_{\theta}^{r}(\Omega) \\ \|u+\varphi\|_{L^{2^{*r}}} \leq 1}} \displaystyle{ \int_{\Omega} {\vert \nabla((-\Delta)^{\frac{r-1}{2}} u) \vert}^{2} }$$ instead of $ S_{\theta,r}(\varphi)= \inf_{ \substack{ u\in H_{\theta}^{r}(\Omega) \\ \|u+\varphi\|_{L^{2^{*r}}} \leq 1}} \displaystyle{ \int_{\Omega} {\vert (-\Delta)^{\frac{r}{2}} u \vert}^{2} } . $  And then we use the same steps to conclude at the end that $ S_{\theta,r}(\varphi)=S_{0,r}(\varphi). $
 This ends the proof of the Theorem \ref{th2}.   \hfill $\square$ \\
\begin{remark}~\\
If $\|\varphi\|_{L^{2^{*r}}}=1$ then $S_{\theta,r}(\varphi)=S_{0,r}(\varphi)=0$ and the infimum are achieved by $0$.\\
Indeed, let $\|\varphi\|_{L^{2^{*r}}} = 1, $ \\
By the definitions in (\ref{eq1}) and (\ref{eq2}) and according to Brezis-Lieb Lemma, for all $ u \in H_{0}^{r}(\Omega) $ we obtain
$$ \|u+\varphi\|_{L^{2^{*r}}} = \|u\|_{L^{2^{*r}}} + \|\varphi\|_{L^{2^{*r}}}+ o(1), $$
that gives, using $\|\varphi\|_{L^{2^{*r}}} = 1, $
$$ \|u\|_{L^{2^{*r}}}= 0 $$
 As a result, we find that $S_{\theta,r}(\varphi)= S_{0,r}(\varphi) = 0$ and the infimum are achieved by $0$.
\end{remark}

\end{document}